\documentclass[reqno]{amsart}
\usepackage{amsmath}
\usepackage[dvips]{graphicx}
\usepackage{amsfonts}
\usepackage{amssymb}
\usepackage{latexsym}
\graphicspath{{Img/}}

\newtheorem{theorem}{Theorem}

\theoremstyle{plain}

\newtheorem{corollary}{Corollary}

\newtheorem{definition}{Definition}
\newtheorem{example}{Example}

\newtheorem{lemma}{Lemma}

\newtheorem{problem}{Problem}
\newtheorem{proposition}{Proposition}
\newtheorem{remark}{Remark}

\numberwithin{equation}{section}

\begin{document}
\title[Solving a problem of angiogenesis]{Solving a problem of angiogenesis of degree three}
\author{Anastasios N. Zachos}
\dedicatory{Dedicated to Professors Dr. Alexander O. Ivanov and
Dr. Alexey A. Tuzhilin for their contributions on minimal trees}
\address{Greece}
\email{azachos@gmail.com}

\keywords{Universal absorbing Fermat-Torricelli set, Universal
Fermat-Torricelli minimum value, generalized Gauss problem,
weighted Fermat-Torricelli problem, weighted Fermat-Torricelli
point, absorbing Fermat-Torricelli tree, absorbing generalized
Gauss tree, evolutionary tree} \subjclass{51E12, 52A10, 52A55,
51E10}
\begin{abstract}

An absorbing weighted Fermat-Torricelli tree of degree four is a
weighted Fermat-Torricelli tree of degree four which is derived as
a limiting tree structure from a generalized Gauss tree of degree
three (weighted full Steiner tree) of the same boundary convex
quadrilateral in $\mathbb{R}^{2}.$ By letting the four variable
positive weights which correspond to the fixed vertices of the
quadrilateral and satisfy the dynamic plasticity equations of the
weighted quadrilateral, we obtain a family of limiting tree
structures of generalized Gauss trees which concentrate to the
same weighted Fermat-Torricelli tree of degree four (universal
absorbing Fermat-Torricelli tree). The values of the residual
absorbing rates for each derived weighted Fermat-Torricelli tree
of degree four of the universal Fermat-Torricelli tree form a
universal absorbing set. The minimum of the universal absorbing
Fermat-Torricelli set is responsible for the creation of a
generalized Gauss tree of degree three for a boundary convex
quadrilateral derived by a weighted Fermat-Torricelli tree of a
boundary triangle (Angiogenesis of degree three). Each value from
the universal absorbing set contains an evolutionary process of a
generalized Gauss tree of degree three.

\end{abstract}\maketitle

\section{Introduction}

We shall describe the structure of a generalized Gauss tree with
degree three and a weighted Fermat-Torricelli tree of degree four
with respect to a boundary convex quadrilateral
$A_{1}A_{2}A_{3}A_{4}$ in $\mathbb{R}^{2}.$

\begin{definition}{\cite[Section~2, pp.~2]{GilbertPollak:68}}\label{topology}
A tree topology is a connection matrix specifying which pairs of
points from the list
$A_{1},A_{2},A_{3},A_{4},A_{0},A_{0^{\prime}}$ have a connecting
linear segment (edge).
\end{definition}

\begin{definition}{\cite[Subsection~1.2, pp.~8]{Ci}\cite{GilbertPollak:68},\cite{IvanovTuzhilin:01}}\label{degreeSteinertree}
The degree of a vertex is the number of connections of the vertex
with linear segments.
\end{definition}

Let $A_{1}, A_{2}, A_{3}, A_{4}$ be four non-collinear points in
$\mathbb{R}^{2}$ and $B_{i}$ be a positive number (weight) which
corresponds to $A_{i}$ for $i=1,2,3,4.$

The weighted Fermat-Torricelli problem for four non-collinear
points (4wFT problem) in $\mathbb{R}^{2}$ states that:

\begin{problem}[4wFT problem]\label{4FTproblemR^2}
Find a point (weighted Fermat-Torricelli point) $A_{0}\in
\mathbb{R}^{2},$ which minimizes
\begin{equation}\label{tetrobjfunction}
f(A_{0})=\sum_{i=1}^{4}B_{i}\|A_{0}-A_{i}\|,
\end{equation}
 where $\|\cdot\|$ denotes the Euclidean distance
\end{problem}

By letting $B_{1}=B_{2}=B_{3}=B_{4}$ in the 4wFT problem we obtain
the following two cases:

(i) If $A_{1}A_{2}A_{3}A_{4}$ is a convex quadrilateral, then
$A_{0}$ is the intersection point of the two diagonals
$A_{1}A_{3}$ and $A_{2}A_{4},$

(ii) If $A_{i}$ is an interior point of $\triangle
A_{j}A_{k}A_{l},$ then $A_{0}\equiv A_{i},$ for $i,j,k,l=1,2,3,4$
and  $i\ne j\ne k\ne l.$

The characterization of the (unique) solution of the 4wFT problem
in $\mathbb{R}^{2}$ is given by the following result which has
been proved in \cite{BolMa/So:99} and \cite{Kup/Mar:97}:

\begin{theorem}{\cite[Theorem~18.37, p.~250]{BolMa/So:99},\cite{Kup/Mar:97}}\label{theor1}

Let $A_{0}$ be a weighted minimum point which minimizes
(\ref{tetrobjfunction}).

(a) Then, the 4wFT point $A_{0}$ uniquely exists.

(b) If for each point $A_{i}\in\{A_{1},A_{2},A_{3},A_{4}\}$

\begin{equation}\label{floatingcase}
\|\sum_{j=1, i\ne j}^{4}B_{j}\vec{u}_{ij}\|>B_i,
\end{equation}
for $i,j=1,2,3,4$  holds, then

($b_{1}$) $A_{0}$ does not belong to $\{A_{1},A_{2},A_{3},A_{4}\}$
 and

($b_{2}$)
\begin{equation}\label{floatingequlcond}
 \sum_{i=1}^{4}B_{i}\vec{u}_{0i}=\vec 0,
\end{equation}
where $\vec{u}_{kl}$ is the unit vector from $A_{k}$ to $A_{l},$
for $k,l\in\{0,1,2,3,4\}$
 (Weighted Floating Case).\\
 (c) If there is a point $A_{i}\in\{A_{1},A_{2},A_{3},A_{4}\}$
 satisfying
 \begin{equation}
 \|{\sum_{j=1,i\ne j}^{4}B_{j}\vec{u}_{ij}}\|\le B_i,
\end{equation}
then $A_{0}\equiv A_{i}.$ (Weighted Absorbed Case).
\end{theorem}

The inverse weighted Fermat-Torricelli problem for four
non-collinear points (Inverse 4wFT problem) in $\mathbb{R}^{2}$
states that:

\begin{problem}{Inverse 4wFT problem}\label{inv4wFT}
Given a point $A_{0}$ which belongs to the convex hull of
$A_{1}A_{2}A_{3}A_{4}$ in $\mathbb{R}^{2}$, does there exist a
unique set of positive weights $B_{i},$ such that
\begin{displaymath}
 B_{1}+B_{2}+B_{3}+B_{4} = c =const,
\end{displaymath}
for which $A_{0}$ minimizes
\begin{displaymath}
 \sum_{i=1}^{4}B_{i}\|\|A_{0}-A_{i}\|.
\end{displaymath}

\end{problem}
By letting $B_{4}=0$ in the inverse 4wFT problem we derive the
inverse 3wfT problem which has been introduced and solved by S.
Gueron and R. Tessler in \cite[Section~4,p.~449]{Gue/Tes:02}.

In 2009, a negative answer with respect to the inverse 4wFT
problem is given in \cite[Proposition~4.4,p.~417]{Zachos/Zou:88}
by deriving a dependence between the four variable weights in
$\mathbb{R}^{2}.$ In 2014, we obtain the same dependence of
variable weights on some $C^{2}$ surfaces in $\mathbb{R}^{3}$ and
we call it the "dynamic plasticity of convex quadrilaterals"
(\cite[Problem~2, Definition~12,
Theorem~1,p.92,p.97-98]{Zachos:14}).


An important generalization of the Fermat-Torricelli problem is
the generalized Gauss problem (or full weighted Steiner tree
problem) for convex quadrilaterals in $\mathbb{R}^{2}$ which has
been studied on the K-plane (Sphere, Hyperbolic plane, Euclidean
plane) in \cite{Zachos:98}.

We mention the following theorem which provide a characterization
for the solutions of the (unweighted )Gauss problem in
$\mathbb{R}^{2}.$

\begin{theorem}{\cite[Theorem(*),pp.~328]{BolMa/So:99}}
Any solution of the Gauss problem is a Gauss tree (equally
weighted full Steiner tree) with at most two (equally weighted)
Fermat-Torricelli points (or Steiner points) where each
Fermat-Torricelli point has degree three, and the angle between
any two edges incident with a Fermat-Torricelli point is of
$120^{\circ}.$
\end{theorem}

We need to mention all the necessary definitions of the weighted
Fermat-Torricelli tree and weighted Gauss tree topologies, in
order to derive some important evolutionary structures of the
Fermat-Torricelli trees (Absorbing Fermat-Torricelli trees) and
Gauss trees (Absorbing Gauss trees) which have been introduced in
\cite[Definitions~1-7,p.~1070-1071]{Zachos:15}.

\begin{definition}\label{FTtopology3}
A weighted Fermat-Torricelli tree topology of degree three is a
tree topology with all boundary vertices of a triangle having
degree one and one interior vertex (weighted Fermat-Torricelli
point) having degree three.
\end{definition}

\begin{definition}\label{FTtopology4}
A weighted Fermat-Torricelli tree topology of degree four is a
tree topology with all boundary vertices of a convex quadrilateral
having degree one and one interior vertex (4wFT point) having
degree four.
\end{definition}

\begin{definition}{\cite[Subsection~3.7,pp.~6]{GilbertPollak:68}}\label{Steinertopology}
A weighted Gauss tree topology (or full Steiner tree topology) of
degree three is a tree topology with all boundary vertices of a
convex quadrilateral having degree one and two interior vertices
(weighted Fermat-Torricelli points) having degree three.
\end{definition}

\begin{definition}\label{Gaussrtree4}
A weighted Fermat-Torricelli tree of weighted minimum length with
a Fermat-Torricelli tree topology of degree four is called a
weighted Fermat-Torricelli tree of degree four.
\end{definition}

\begin{definition}{\cite[Subsection~3.7,pp.~6]{GilbertPollak:68},\cite{Zachos:15}}\label{Gaussrtree3}
A weighted Gauss tree of weighted minimum length with a Gauss tree
topology of degree three is called a generalized Gauss tree of
degree three or a full weighted minimal Steiner tree.
\end{definition}

In 2014, we study an important generalization of the weighted
Gauss (tree) problem that we call a generalized Gauss problem for
convex quadrilaterals in $\mathbb{R}^{2}$ by using a mechanical
construction which extends the mechanical construction of
Gueron-Tessler in the sense of P$\acute{o}$lya and Varigon
(\cite[Problem~1, Theorem~4, pp.1073-1075]{Zachos:15}).


We state a generalized Gauss problem for a weighted convex
quadrilateral $A_{1}A_{2}A_{3}A_{4}.$ in $\mathbb{R}^{2},$ such
that the weights $B_{i}$ which correspond to $A_{i}$ and
$B_{00^{\prime}}\equiv x_{G},$ satisfy the inequalities

\[|B_{i}-B_{j}|<B_{k}<B_{i}+B_{j},\]

and

\[|B_{t}-B_{m}|<B_{n}<B_{t}+B_{m}\]

where $x_{G}$ is the variable weight which corresponds to the
given distance $l\equiv \|A_{0}-A_{0^{\prime}}\|,$ for
$i,j,k\in\{1,4,00^{\prime}\},$ $t,m,n\in\{2,3,00^{\prime}\}$ and
$i\ne j\ne k,$ $t\ne m\ne n.$

\begin{problem}{\cite[Problem~1,p.~1073]{Zachos:15}}\label{genGaussproblem}
Given $l,$ $B_{1},$ $B_{2},$ $B_{3},$ $B_{4},$ find a generalized
Gauss tree of degree three with respect to $A_{1}A_{2}A_{3}A_{4}$
which minimizes

\begin{equation} \label{fundamentaleqxGauss}
B_{1}\|A_{1}-A_{0}\|+B_{2}\|A_{2}-A_{0}\|+B_{3}\|A_{3}-A_{0^{\prime}}\|+B_{4}\|A_{4}-A_{0^{\prime}}\|+x_{G}l.
\end{equation}

\end{problem}

For $l=0,$ we obtain a weighted Fermat-Torricelli tree of degree
four.

\begin{definition}{\cite[Definition~8, p.1076]{Zachos:15}}\label{varGauss} We call the variable $x_{G}$
which depend on $l,$ a generalized Gauss variable.
\end{definition}

\begin{definition}{\cite[Definition~9, p.1080]{Zachos:15}}\label{varGauss}\label{absorbingrate}
The residual absorbing rate of a generalized Gauss tree of degree
at most four with respect to a boundary convex quadrilateral is
\[\sum_{i=1}^{4}B_{i}-x_{G}.\]
\end{definition}

\begin{definition}{\cite[Definition~10, p.1080]{Zachos:15}}\label{AbsorbingGeneralizedtree}
An absorbing generalized Gauss tree of degree three is a
generalized Gauss tree of degree three with residual absorbing
rate
\[\sum_{i=1}^{4}B_{i}-x_{G}.\]
\end{definition}


In this paper, we prove that the weighted Fermat-Torricelli
problem for convex quadrilaterals cannot be solved analytically,
by extending the geometrical solution of E. Torricelli and
E.Engelbrecht for convex quadrilaterals in $\mathbb{R}^{2}$
(Section~2, Theorems~3,4)

In section~3, we derive a solution for the generalized Gauss
problem in $\mathbb{R}^{2},$ in the spirit of K. Menger which
depends only on five given positive weighta and five Euclidean
distances which determine a convex quadrilateral (Section~3,
Theorem~5).

In section~4, We give a new approach concerning the dynamic
plasticity of quadrilaterals by deriving a new system of equations
different from the dynamic plasticity equations which have been
deduced in \cite[Proposition~4.4,p.~417]{Zachos/Zou:88} and
\cite[Problem~2, Definition~12, Theorem~1,p.92,p.97-98]{Zachos:14}
(Section~4, Proposition~2). Furthermore, we obtain a surprising
connection of the dynamic plasticity of quadrilaterals with a
problem of Rene Descarted posed in 1638.

In section~5, we introduce an absorbing weighted Fermat-Torricelli
tree of degree four which is derived as a limiting tree structure
from a generalized Gauss tree of degree three of the same boundary
convex quadrilateral in $\mathbb{R}^{2}.$

Then, by assuming that the four variable positive weights which
correspond to the fixed vertices of the boundary quadrilateral and
satisfy the dynamic plasticity equations, we obtain a family of
limiting tree structures of generalized Gauss trees of degree
three which concentrate to the same weighted Fermat-Torricelli
tree of degree four in a geometric sense (Universal absorbing
Fermat-Torricelli tree).

Furthermore, we calculate the values of the universal rates of a
Universal absorbing tree regarding a fixed boundary quadrilaterals
(Section~5, Examples~2,3).

In section~6, we introduce a class of Euclidean minimal tree
structures that we call steady trees and evolutionary trees
(Section~6, Definitions~15,16). Thus, the minimum of the universal
absorbing Fermat-Torricelli set (Universal Fermat-Torricelli
minimum value) leads to the creation of a generalized Gauss tree
of degree three for the same boundary convex quadrilateral which
is derived by a weighted Fermat-Torricelli tree of degree four. A
universal absorbing Fermat-Torricelli minimum value corresponds to
the intersection point (4wFT point). This quantity is of
fundamental importance, because by attaining this value the
absorbing Fermat-Torricelli tree start to grow and will be able to
produce a generalized Gauss tree of degree three (Evolutionary
tree). Each specific value from the universal absorbing set gives
an evolutionary process of a generalized Gauss tree of degree
three regarding a fixed boundary quadrilateral by spending a
positive quantity from the storage of the universal
Fermat-Torricelli quantity which stimulates the evolution at the
4wFT point (Section~6, Example~4, Angiogenesis of degree three).


\section{Extending Torricelli-Engelbrecht's solution for convex quadrilaterals}\label{sec1}

The weighted Torricelli-Engelbrecht solution for a triangle
$\triangle A_{1}A_{2}A_{3}$ in the weighted floating case is given
by the following proposition:

\begin{lemma}{\cite{Gue/Tes:02},\cite{ENGELBRECHt:1877}}\label{trianglesolveFT}
If $A_{0}$ is an interior weighted Fermat-Torricelli point of
$\triangle A_{1}A_{2}A_{3},$ then
\begin{equation}
\angle
A_{i}A_{0}A_{j}\equiv\alpha_{i0j}=\arccos{\left(\frac{B_{k}^{2}-B_{i}^{2}-B_{j}^{2}}{2
B_{i}B_{j}}\right)}.
\end{equation}
\end{lemma}

Let  $A_{1}A_{2}A_{3}A_{4}$ be a convex quadrilateral in
$\mathbb{R}^{2},$ $O$ be the intersection point of the two
diagonals and $B_{i}$ be a given weight which corresponds to the
vertex $A_{i},$ $A_{0}$ be the weighted Fermat-Torricelli point in
the weighted floating case (Theorem~) and $\vec{U}_{ij}$ be the
unit vector from $A_{i}$ to $A_{j},$ for $i,j=1,2,3,4.$

We mention the geometric plasticity principle of quadrilaterals in
$\mathbb{R}^{2},$

\begin{lemma}{\cite{Zachos/Zou:88},\cite[Definition~13,Theorem~3, Proposition~8, Corollary~4 p.~103-108]{Zachos:14}}\label{geomplasticityR2}
Suppose that the weighted floating case of the weighted
Fermat-Torricelli point $A_{0}$ point with respect to
$A_{1}A_{2}A_{3}A_{4}$ is satisfied:
\[\left\|
B_{i}\vec{U}_{ki}+B_{j}\vec{U}_{kj}+B_{m}\vec{U}_{km}\right\|>
B_{k},\] for each $i,j,k,m=1,2,3,4$ and $i\ne j\ne k\ne m.$ If
$A_{0}$ is connected with every vertex $A_{k}$ for $k=1,2,3,4$ and
we select a point $A_{k}^{\prime}$ with non-negative weight
$B_{k}$ which lies on the ray $A_{k}A_{0}$ and the quadrilateral
$A_{1}^{\prime}A_{2}^{\prime}A_{3}^{\prime}A_{4}^{\prime}$ is
constructed such that:
\[\left\|
B_{i}\vec{U}_{k^{\prime}i^{\prime}}+B_{j}\vec{U}_{k^{\prime}j^{\prime}}+B_{m}\vec{U}_{k^{\prime}m^{\prime}}\right\|>
B_{k},\] for each
$i^{\prime},j^{\prime},k^{\prime},m^{\prime}=1,2,3,4$ and
$i^{\prime}\ne j^{\prime}\ne k^{\prime}\ne m^{\prime}.$
 Then the weighted Fermat-Torricelli point
$A_{0}^{\prime}$ is identical with $A_{0}.$
\end{lemma}

\begin{theorem}

The weighted Torricelli-Engelbrecht solution for
$A_{1}A_{2}A_{3}A_{4}$ is given by the following system of four
equations w.r. to the variables $\alpha_{102},$ $\alpha_{203},$
$\alpha_{304}$ and $\alpha_{401}:$
\begin{eqnarray}\label{first}
& \csc^{2}\alpha_{102} \csc^{2}\alpha _{304} \csc^{2}\alpha _{401}
\left(\cos\alpha_{102}-\sin\alpha_{102}\right) \left(\cos\alpha
_{304}-\sin\alpha
_{304}\right)\nonumber \\
{}& \left(\cos\left(\alpha_{102}-\alpha
_{304}\right)-\cos\left(\alpha_{102}+\alpha _{304}+2 \alpha
_{401}\right)-2 \sin\left(\alpha_{102}+\alpha
_{304}\right)\right)=0,
\end{eqnarray}
\begin{eqnarray}\label{second}
-B_1^2-2B_1 B_2\cos\alpha _{102}-B_2^2+B_3^2-2B_1 B_4 \cos\alpha
_{401} -2B_2 B_4 \cos\left(\alpha _{102}+\alpha _{401}\right)
-B_4^2=0,
\end{eqnarray}
\begin{eqnarray}\label{third}
\alpha_{304}=\arccos{\left(\frac{B_1^2+2B_1 B_2\cos\alpha
_{102}+B_2^2-B_3^2-B_4^2}{2 B_3 B_4}\right)}
\end{eqnarray}
and
\begin{eqnarray}\label{fourth}
\alpha_{203}=2\pi-\alpha_{102}-\alpha_{304}-\alpha_{401}.
\end{eqnarray}

\end{theorem}

\begin{proof}

Assume that we select $B_{1},$ $B_{2},$ $B_{3},$ $B_{4},$ such
that $A_{0}$ is an interior point of $\triangle A_{1}OA_{2}.$ By
applying the geometric plasticity principle of
Lemma~\ref{geomplasticityR2} we could choose a transformation of
$A_{1}A_{2}A_{3}A_{4}$ to the square
$A_{1}^{\prime}A_{2}^{\prime}A_{3}^{\prime}A_{4}^{\prime},$ where
$A_{0}^{\prime}=A_{0}$ and $A_{O}$ is an interior point of
$A_{1}^{\prime}O^{\prime}A_{2}^{\prime}$ where $O^{\prime}$ is the
intersection of $A_{1}^{\prime}A_{3}^{\prime}$ and
$A_{2}^{\prime}A_{4}^{\prime}.$

We consider the equations of the three circles which pass through
$A_{1},$ $A_{2},$ $A_{0},$ $A_{1},$ $A_{4},$ $A_{0}$ and $A_{3},$
$A_{4},$ $A_{0},$ respectively, which meet at $A_{0}:$

\begin{eqnarray}\label{circle1}
\left(x-\frac{a}{2}\right)^2+\left(y-\frac{1}{2} a \cot\alpha
_{102}\right)^2=\frac{1}{4} a^2 \csc^{2}\alpha _{102}
\end{eqnarray}

\begin{eqnarray}\label{circle2}
\left(y-\frac{a}{2}\right)^2+\left(x-\frac{1}{2} a \cot\alpha
_{401}\right)^2=\frac{1}{4} a^2 \csc^{2}\alpha _{401}
\end{eqnarray}

\begin{eqnarray}\label{circle3}
\left(x-\frac{a}{2}\right)^2+\left(y-(a-\frac{1}{2} a \cot\alpha
_{304})\right)^2=\frac{1}{4} a^2 \csc^{2}\alpha _{304}
\end{eqnarray}

By subtracting (\ref{circle2}) from (\ref{circle1}),
(\ref{circle3}) from (\ref{circle1}) and solving w.r. to $x, y$ we
get:

\begin{eqnarray}\label{countx}
x=\frac{a(-1+\cot\alpha_{102})(-1+\cot\alpha_{304})}{(-2+\cot\alpha_{102}+\cot\alpha_{304})
(-1+\cot\alpha_{401})}
\end{eqnarray}
and
\begin{eqnarray}\label{county}
y=-\frac{-a
\cot\alpha_{304}+a}{\cot\alpha_{102}+\cot\alpha_{304}-2}
\end{eqnarray}

By substituting (\ref{countx}) and (\ref{county}) in
(\ref{circle1}), we obtain (\ref{first}).

Taking into account the weighted floating equilibrium condition,
we get:

\begin{eqnarray}\label{vectorbalance1}
-B_{3} \vec{u}_{30}=B_{1} \vec{u}_{10}+B_{2} \vec{u}_{20}+B_{4}
\vec{u}_{40}
\end{eqnarray}

or

\begin{eqnarray}\label{vectorbalance2}
B_{1} \vec{u}_{10}+B_{2} \vec{u}_{20}=-B_{3}
\vec{u}_{30}-B_{4}\vec{u}_{40}.
\end{eqnarray}

By squaring both parts of (\ref{vectorbalance1}) we derive
(\ref{second}) and by squaring both parts of
(\ref{vectorbalance2}) we derive (\ref{third}).

\end{proof}

\begin{remark}\label{rem1}
A different approach was used in \cite[Solution~2.2,Example~2.4,
p.~413-414]{Zachos/Zou:88}, in order to derive a similar system of
equations w.r. to $\alpha_{401}$ and $\alpha_{102}.$
\end{remark}

\begin{example}\label{exam1}
By substituting $B_{1}=3.5,$ $B_{2}=2.5,$ $B_{3}=2,$ $B_{4}=1,$
$a=10$ in (\ref{first}) and (\ref{second}) and solving this system
of equations numerically by using for instance Newton method and
choosing as initial values $\alpha_{102}^{o}=2.7$ rad,
$\alpha_{401}^{o}=1.2$ rad, we obtain $\alpha_{102}=2.30886$ and
$\alpha_{401}=1.57801$ rad. By substituting $\alpha_{102}=2.30886$
and $\alpha_{401}=1.57801$ rad in (\ref{third}) we get
$\alpha_{304}=1.12492$ rad. From (\ref{fourth}), we get
$\alpha_{203}=1.2714$ rad. By substituting the angles
$\alpha_{i0j}$ in (\ref{countx}) and (\ref{county}), we derive
$x=4.0700893$ and $y=2.146831.$
\end{example}

\begin{theorem}\label{nonexist4FTsol}
There does not exist an analytical solution for the 4wFT problem
in $\mathbb{R}^{2}.$
\end{theorem}

\begin{proof}

The system of the two equations (\ref{first}) and (\ref{second})
taking into account (\ref{third}) cannot be solved explicitly w.r
to $\alpha_{102}$ and $\alpha_{401}.$ Therefore, by considering
(\ref{countx}) and (\ref{county}) we deduce that the location of
the 4wFT point $A_{0}$ cannot also be expressed explicitly via the
angles $\alpha_{i0j},$ for $i,j=1,2,3,4,$ for $i\ne j.$
\end{proof}

Thus, from Theorem~\ref{nonexist4FTsol} the position of a weighted
Fermat-Torricelli tree of degree four cannot be expressed
analytically and may be found by using numerical methods (see also
in \cite{Zachos/Zou:88}).

\section{An absorbing generalized Gauss-Menger tree in $\mathbb{R}^{2}$}

Let $A_{1}A_{2}A_{3}A_{4}$ be a boundary weighted convex
quadrilateral of a generalized Gauss tree of degree three in
$\mathbb{R}^{2}$ and $A_{0},$ $A_{0^{\prime}}$ are the two
weighted Fermat-Torricelli (3wFT) points of degree three which are
located at the convex hull of the boundary quadrilateral.

We denote by $l\equiv \|A_{0}-A_{0^{\prime}}\|,$ $a_{ij}\equiv
\|A_{i}-A_{j}\|,$ $\alpha_{ijk}\equiv \angle A_{i}A_{j}A_{k},$
$a_{10}\equiv a_{1},$ $a_{40}\equiv a_{4},$ $a_{20^{\prime}}\equiv
a_{2},$ $a_{30^{\prime}}\equiv a_{3}$ (Fig.~\ref{fig1n1}) and by
$B_{i}\equiv \frac{B_{i}^{\prime}}{\sum_{i=1}^{4}B_{i}^{\prime}},$
for $i, j, k\in\{0,0^{\prime},1,2,3,4\}$ and $i\ne j\ne k.$

\begin{figure}
\centering
\includegraphics[scale=0.75]{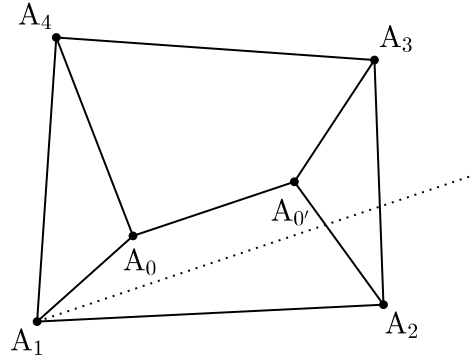}
\caption{A generalized Gauss Menger tree of degree three regarding
a boundary convex quadrilateral}\label{fig1n1}
\end{figure}

We proceed by giving the following two lemmas which have been
proved recently in \cite[Theorem~1]{Zachos:98} and
\cite{Zachos/Zou:13}.

\begin{lemma}{\cite[Theorem~1]{Zachos:98},\cite[Theorem~2.1,p.~485]{Zachos/Zou:13}}\label{FT K quadrilateral}

A generalized Gauss tree of degree three (full weighted Steiner
minimal tree) of $A_{1}A_{2}A_{3}A_{4}$ consists of two weighted
Fermat-Torricelli points $A_{0},$ $A_{0}^{\prime}$ which are
located at the interior convex domain with corresponding given
weights $B_{0},$ $B_{0^{\prime}}$ and minimizes the objective
function:
\begin{equation} \label{eq:B_1}
B_{1}a_{1}+B_{2}a_{2}+B_{3}a_{3}+B_{4}a_{4}+\frac{B_{0}+B_{0^{\prime}}}{2}l\to
min,
\end{equation}
such that:
\begin{equation}\label{ineq1}
|B_{i}-B_{j}|<B_{k}<B_{i}+B_{j},
\end{equation}
\begin{equation}\label{ineq2}
|B_{t}-B_{m}|<B_{n}<B_{t}+B_{m}
\end{equation}

where \[B_{00^{\prime}}\equiv\frac{B_{0}+B_{0^{\prime}}}{2},\]

for $i,j,k\in\{1,4,00^{\prime}\},$ $t,m,n\in\{2,3,00^{\prime}\}$
and $i\ne j\ne k,$ $t\ne m\ne n,$
\end{lemma}

Suppose that $B_{1},$ $B_{2},$ $B_{3},$ $B_{4},$ $B_{00^{\prime}}$
satisfy the inequalities (\ref{ineq1}) and (\ref{ineq2}).

\begin{lemma}{\cite[Theorem~2.2, p.~486]{Zachos/Zou:13}}\label{explicitsolution}
The location of $A_{0}$ and $A_{0^{\prime}}$ is given by the
relations:

\begin{equation}\label{base1}
\cot{\varphi}=\frac{B_{0}
a_{12}+B_{4}a_{14}\cos(\alpha_{214}-\alpha_{400^{\prime}})+B_{3}a_{23}\cos(\alpha_{123}-\alpha_{30^{\prime}0})}{B_{4}a_{14}\sin(\alpha_{214}-\alpha_{400^{\prime}})-B_{3}a_{23}\sin(\alpha_{123}-\alpha_{30^{\prime}0})},
\end{equation}

\begin{equation}\label{a1solvesystem}
a_{1}=\frac{a_{14}\sin(\alpha_{214}-\varphi-\alpha_{400^{\prime}})}{\sin(\alpha_{100^{\prime}}+\alpha_{400^{\prime}})}
\end{equation}

and

\begin{equation}\label{a2solvesystem}
a_{2}=\frac{a_{23}\sin(\alpha_{123}+\varphi-\alpha_{300^{\prime}})}{\sin(\alpha_{20^{\prime}0}+\alpha_{30^{\prime}0})}
\end{equation}

 where $\varphi$ is the angle which is formed between the
line defined by $A_{1}$ and $A_{2}$ and the line which passes from
$A_{1}$ and it is parallel to the line defined
 by $A_{0}$ and $A_{0}^{\prime}.$
\end{lemma}

\begin{definition}\label{GaussMengersolution} A generalized
Gauss-Menger tree is a solution of a generalized Gauss problem in
$\mathbb{R}^{2}$ for a boundary quadrilateral
$A_{1}A_{2}A_{3}A_{4}$ which depend on the Euclidean distances
$a_{ij}$ and the five given weights $B_{1},$ $B_{2},$ $B_{3},$
$B_{4}$ and $B_{00^{\prime}}.$
\end{definition}

By letting $B_{00^{\prime}}\equiv x_{G},$ we obtain an absorbing
generalized Gauss-Menger tree for a boundary quadrilateral.

\begin{theorem}\label{DependenceGaussmenger}
An absorbing generalized Gauss-Menger tree w.r. to a fixed convex
quadrilateral $A_{1}A_{2}A_{3}A_{4}$ depends only on the five
given weights $B_{i},$ $B_{2},$ $B_{3},$ $B_{4},$
$B_{00^{\prime}}\equiv x_{G}$ and the five given lengths $a_{12},$
$a_{23},$ $a_{34},$ $a_{41}$ and $a_{13}.$
\end{theorem}

\begin{proof}

Consider the Caley-Menger determinant which gives the volume of a
tetrahedron $A_{1}A_{2}A_{3}A_{4}$ in $\mathbb{R}^{3}.$

\begin{equation}\label{CaleyMengertetr}
288 V^{2} =\operatorname{det} \left(
\begin{array}{ccccc}
0      & a_{12}^{ 2}      & a_{13}^{ 2} & a_{14}^{ 2}  &  1  \\
a_{12}^{ 2} & 0 & a_{23}^{ 2} &a_{24}^{ 2} & 1          \\
a_{13}^{ 2} & a_{23}^{ 2} &0   &a_{34}^{ 2}  & 1         \\
a_{14}^{ 2} & a_{24}^{ 2} &a_{34}^{ 2}  & 0   & 1         \\
1 & 1 &1  & 1     &   0    \\
\end{array} \right).
\end{equation}

By letting $V=0$ in (\ref{CaleyMengertetr}), we obtain a
dependence of the six distances $a_{12},$ $a_{13},$ $a_{14},$
$a_{23},$ $a_{34}$ and $a_{24}.$ For instance, by solving a fourth
order degree equation w.r. to $a_{13}$ we derive that
$a_{13}=(a_{12},a_{14},a_{23},a_{34},a_{24}).$

By applying the cosine law in $\triangle A_{1}A_{2}A_{4}$ and
$\triangle A_{1}A_{2}A_{3}$ we get:

\begin{equation}\label{cosalapha214}
\alpha_{214}=\arccos \left(\frac{a_{12}^{ 2}+a_{14}^{ 2}-a_{24}^{
2}}{2 a_{12} a_{14} }\right),
\end{equation}

and

\begin{equation}\label{cosalapha123}
\alpha_{123}=\arccos \left(\frac{a_{12}^{ 2}+a_{23}^{ 2}-a_{13}^{
2}}{2 a_{12} a_{23} }\right).
\end{equation}

By Lemma~\ref{trianglesolveFT} and Lemma~\ref{FT K quadrilateral}
and considering that $A_{0}$ is the 3wFT point of $\triangle
A_{1}A_{4}A_{0^{\prime}}$ and $A_{0^{\prime}}$ is the 3wFT point
of $\triangle A_{2}A_{3}A_{0}$ we get:

\begin{equation}\label{ft01}
\alpha_{100^{\prime}}=\arccos\left(\frac{B_{4}^{2}-B_{1}^{2}-x_{G}^{2}}{2
B_{1}x_{G}}\right),
\end{equation}
\begin{equation}\label{ft02}
\alpha_{0^{\prime}04}=\arccos\left(\frac{B_{1}^{2}-B_{4}^{2}-x_{G}^{2}}{2
B_{4}x_{G}}\right),
\end{equation}
\begin{equation}\label{ft03}
\alpha_{104}=\arccos\left(\frac{x_{G}^{2}-B_{1}^{2}-B_{4}^{2}}{2
B_{1}B_{4}}\right),
\end{equation}

\begin{equation}\label{ft021}
\alpha_{00^{\prime}3}=\arccos\left(\frac{B_{2}^{2}-B_{3}^{2}-x_{G}^{2}}{2
B_{3}x_{G}}\right),
\end{equation}
\begin{equation}\label{ft022}
\alpha_{00^{\prime}2}=\arccos\left(\frac{B_{3}^{2}-x_{G}^{2}-B_{2}^{2}}{2
x_{G}B_{2}}\right),
\end{equation}

and
\begin{equation}\label{ft023}
\alpha_{20^{\prime}3}=\arccos\left(\frac{x_{G}^{2}-B_{2}^{2}-B_{3}^{2}}{2
B_{2}B_{3}}\right).
\end{equation}

Therefore, by replacing (\ref{ft01}), (\ref{ft02}), (\ref{ft022}),
(\ref{ft021}), (\ref{cosalapha214}), (\ref{cosalapha123}) in
(\ref{base1}), (\ref{a1solvesystem}), (\ref{a2solvesystem}) and
taking into account the dependence of the six distances $a_{ij},$
for $i,j=1,2,3,4,$ we derive that $\varphi,$ $a_{1}$ and $a_{2}$
depend only on $B_{1},$ $B_{2},$ $B_{3},$ $B_{4},$ $x_{G}$ and
$a_{12},$ $a_{13},$ $a_{23},$ $a_{34},$ $a_{24}.$

\end{proof}


\section{The dynamic plasticity of convex quadrilaterals}

In this section, we deal with the solution of the inverse 4wFT
problem in $\mathbb{R}^{2}$ which has been introduced in
\cite{Zachos/Zou:88} and developed in \cite{Zachos:14}, in order
to obtain a new system of equations of the dynamic plasticity of
weighted quadrilaterals w.r. to the four variable weights
$(B_{i})_{1234}$ for $i=1,2,3,4,$ which cover also the case
$(B_{1})_{1234}=(B_{3})_{1234}$ and
$(B_{2})_{1234}=(B_{4})_{1234}.$

First, we start by mentioning the solution of S. Gueron and R.
Tessler (\cite[Section~4,p.~449]{Gue/Tes:02}) of the inverse 3wFT
problem for three non-collinear points in $\mathbb{R}^{2}.$ By
letting $B_{i}=0$ in the inverse 4wFT problem for convex
quadrilaterals (Problem~\ref{inv4wFT}) we derive the  inverse 3wFT
problem for s triangle.

Consider the inverse 3wFT problem for $\triangle A_{i}A_{j}A_{k}$
in $\mathbb{R}^{2}.$

\begin{lemma}{\cite[Section~4,p.~449]{Gue/Tes:02}}\label{inv3imp} The unique solution of the
inverse 3wFT problem for $\triangle A_{i}A_{j}A_{k}$ is given by

\begin{equation}\label{solinv3FT}
(\frac{B_{i}}{B_{j}})_{ijk}=\frac{\sin\alpha_{jik}}{\sin\alpha_{ijk}}.
\end{equation}

\end{lemma}

\begin{definition}{\cite{Zachos:14}} We call dynamic plasticity of a weighted Fermat-Torricelli tree of degree
four the set of solutions of the four variable weights with
respect to the inverse 4wFT problem in $\mathbb{R}^{2}$ for a
given constant value $c$ which correspond to a family of weighted
Fermat-Torricelli tree of degree four that preserve the same
Euclidean tree structure (the corresponding 4wFT point remains the
same for a fixed boundary convex quadrilateral), such that the
three variable weights depend on a fourth variable weight and the
value of $c.$
\end{definition}

By taking into account Lemma~\ref{inv3imp} for the triangles
$\triangle A_{1}A_{2}A_{3},$ $\triangle A_{1}A_{3}A_{4},$
$\triangle A_{1}A_{2}A_{4}$ and the weighted floating equilibrium
condition (\ref{floatingequlcond}) taken from Theorem~\ref{theor1}
(\cite{Zachos/Zou:88}, \cite{Zachos:14})

\begin{proposition}{\cite[Proposition~4.4,p.~417]{Zachos/Zou:88},\cite[Problem~2, Definition~12,
Theorem~1,p.92,p.97-98]{Zachos:14}}\label{dynamicplasticityR2}
Suppose that $A_{0}$ does not belong to the intersection of the
linear segments $A_{1}A_{3}$ and $A_{2}A_{4}.$ The dynamic
plasticity of the variable weighted Fermat-Torricelli tree of
degree four in $\mathbb{R}^{2}$ is given by the following three
equations:
\begin{eqnarray} \label{plastic1}
 (\frac{B_2}{B_1})_{1234}=(\frac{B_2}{B_1})_{123}[1-(\frac{B_4}{B_1})_{1234}(\frac{B_1}{B_4})_{134}],\\
\label{plastic2}
 (\frac{B_3}{B_1})_{1234}=(\frac{B_3}{B_1})_{123}[1-(\frac{B_4}{B_1})_{1234}(\frac{B_1}{B_4})_{124}],
\end{eqnarray}
and
\begin{equation}\label{invcond4}
 (B_{1})_{1234}+(B_{2})_{1234}+(B_{3})_{1234}+(B_{4})_{1234}=c=const.
\end{equation}
\end{proposition}

It is worth mentioning that each quad of values
$\{(B_{1})_{1234},(B_{2})_{1234},(B_{3})_{1234},
(B_{4})_{1234}\},$ which satisfy simultaneously (\ref{plastic1}),
(\ref{plastic2}), (\ref{invcond4}) create a unique concentration
of different families of tetrafocal ellipses (Polyellipse or
Egglipse) to the same 4wFT point $A_{0}$ of
$A_{1}A_{2}A_{3}A_{4},$ A family of tetrafocal ellipses may be
constructed by selecting a decreasing sequence of real numbers
$c_{n}((B_{1})_{1234},(B_{2})_{1234},(B_{3})_{1234},(B_{4})_{1234}))$

\[c_{n}((B_{1})_{1234},(B_{2})_{1234},(B_{3})_{1234},(B_{4})_{1234});X)=\sum_{i=1}^{4}(B_{i})_{1234}\|A_{i}-X\|\]

which converge to the constant number
\[c((B_{1})_{1234},(B_{2})_{1234},(B_{3})_{1234},(B_{4})_{1234})\equiv
f(A_{0},(B_{1})_{1234},(B_{2})_{1234},(B_{3})_{1234},(B_{4})_{1234}).\]

or
\[c((B_{1})_{1234},(B_{2})_{1234},(B_{3})_{1234},(B_{4})_{1234})=\sum_{i=1}^{4}(B_{i})_{1234}\|A_{i}-A_{0}\|\]
or
\[c((B_{1})_{1234},(B_{2})_{1234},(B_{3})_{1234},(B_{4})_{1234})=\sum_{i=1}^{4}(B_{i})_{1234}a_{i}.\]
The concentration of different families of tetrafocal ellipses to
the same point provide a surprising connection with a problem
posed by R. Descartes in 1638. According to
\cite[Chapter~II,p.~235]{BolMa/So:99}, in a letter from August 23,
1638, R. Descartes invited P. de Fermat to investigate the
properties of tetrafocal ellipses in $\mathbb{R}^{2}.$ The dynamic
plasticity of weighted quadrilaterals solves the problem of
concentration of tetrafocal ellipse and offers a new property to
R.Descartes' problem.

We proceed by deriving a new system of dynamic plasticity
equations for a weighted Fermat-Torricelli tree of degree four
which also includes a class of weighted Fermat-Torricelli trees of
degree four which coincides with the two diagonals of the boundary
convex quadrilateral for $(B_{1})_{1234}=(B_{3})_{1234}$ and
$(B_{2})_{1234}=(B_{4})_{1234}.$

\begin{proposition}\label{dynamicplasticityR2n}
The dynamic plasticity of the variable weighted Fermat-Torricelli
tree of degree four in $\mathbb{R}^{2}$ is given by the following
three equations:
\begin{eqnarray} \label{plastic14n}
(B_{1})_{1234}^{2}+ (B_{2})_{1234}^{2}+2 (B_{1})_{1234}
(B_{2})_{1234}\cos\alpha_{102}= \nonumber\\
(B_{3})_{1234}^{2}+(B_{4})_{1234}^{2}+2 (B_{3})_{1234}
(B_{4})_{1234}\cos\alpha_{304},
\end{eqnarray}
\begin{eqnarray}\label{plastic24n}
(B_{1})_{1234}^{2}+ (B_{4})_{1234}^{2}+2 (B_{1})_{1234}
(B_{4})_{1234}\cos\alpha_{104}=\nonumber\\ (B_{2})_{1234}^{2}+
(B_{3})_{1234}^{2}+2 (B_{2})_{1234}
(B_{3})_{1234}\cos\alpha_{203},
\end{eqnarray}
and
\begin{equation}\label{invcond4n}
 (B_{1})_{1234}+(B_{2})_{1234}+(B_{3})_{1234}+(B_{4})_{1234}=c const.
\end{equation}
\end{proposition}

\begin{proof}

Suppose that we select four positive weights $(B_{i})_{1234}(0)$
which correspond the vertex $A_{i}$ of the boundary convex
quadrilateral $A_{1}A_{2}A_{3}A_{4},$ such that the weighted
floating inequalities (\ref{floatingcase}) of Theorem~\ref{theor1}
hold, in order to locate the 4wFT point $A_{0}$ at the interior of
$A_{1}A_{2}A_{3}A_{4}.$

From the weighted floating (variable weighted) equilibrium
condition of the 4wFT point (\ref{floatingequlcond})we get

\begin{equation}\label{vect1}
(B_{1})_{1234}\vec{u}_{10}+(B_{2})_{1234}\vec{u}_{20}=-(B_{3})_{1234}\vec{u}_{30}-(B_{4})_{1234}\vec{u}_{40}
\end{equation}

and

\begin{equation}\label{vect2}
(B_{1})_{1234}\vec{u}_{10}+(B_{3})_{1234}\vec{u}_{30}=-(B_{2})_{1234}\vec{u}_{20}-(B_{4})_{1234}\vec{u}_{40},
\end{equation}

which yield (\ref{plastic14n}) and (\ref{plastic24n}),
respectively.

\end{proof}

\begin{corollary}\label{two equal}
If the variable weighted 4wFT point is the intersection point of
$A_{1}A_{3}$ and $A_{2}A_{4},$ then

\begin{equation}\label{2eq1}
(B_{1})_{1234}=(B_{3})_{1234}
\end{equation}

and

\begin{equation}\label{2eq2}
(B_{2})_{1234}=(B_{4})_{1234}.
\end{equation}

$(B_{1})_{1234}=(B_{3})_{1234}$ and
$(B_{2})_{1234}=(B_{4})_{1234}.$
\end{corollary}

\begin{proof}
By letting $\alpha_{102}=\alpha_{304}$ and
$\alpha_{104}=\alpha_{203},$ in (\ref{plastic14n}) and
(\ref{plastic24n}) we obtain (\ref{2eq1}) and (\ref{2eq2}),
respectively.
\end{proof}

\begin{remark}

The dynamic plasticity equations of
Theorem~\ref{dynamicplasticityR2} depend on the solutions of the
inverse 3FT problem for $\triangle A_{1}A_{2}A_{3},$ $\triangle
A_{1}A_{3}A_{4}$ and $\triangle A_{1}A_{2}A_{4}.$ Thus, the
corresponding 4wFT point $A_{0}$ remains at the interior of
$\triangle A_{1}A_{2}A_{3},$ in order to derive from the inverse
3wFT problem for $\triangle A_{1}A_{2}A_{3}$ the inverse wFT
problem for $A_{1}A_{2}A_{3}A_{4}.$ The dynamic plasticity
equations of Theorem~\ref{dynamicplasticityR2n} generalizes the
dynamic plasticity equations of Theorem~\ref{dynamicplasticityR2}
because $A_{0}$ could also lie on the side $A_{1}A_{3}$ of
$\triangle A_{1}A_{2}A_{3}$ and the side $A_{2}A_{4}$ of
$\triangle A_{1}A_{2}A_{4}.$ These are the cases where the inverse
3wFT problem for $\triangle A_{1}A_{2}A_{3}$ and $\triangle
A_{1}A_{2}A_{4}$ do not hold.

\end{remark}


\section{A universal Fermat-Torricelli minimal value of a family of absorbing generalized Gauss trees of degree three}

Suppose that an absorbing weighted Fermat-Torricelli tree of
degree four is derived as a limiting tree structure from an
absorbing generalized Gauss-Menger tree of degree three regarding
a fixed boundary quadrilateral $A_{1}A_{2}A_{3}A_{4}$ for a
specific value $c_{G}$ of the generalized Gauss variable.

We need to consider the following lemma which gives $\bar{B_{i}}$
as a linear function of $\bar{B_{4}}$ regarding a fixed variable
weighted Fermat-Torricelli tree of degree four in
$\mathbb{R}^{2}.$

\begin{lemma}{\cite[Corollary~4.5,p.~418]{Zachos/Zou:88}}\label{lemlinearB4}Let
$\sum_{1234}^{}\bar{B}:=(\bar{B_1})_{1234}(1+\frac{\bar{B_2}}{\bar{B_1}}+\frac{\bar{B_3}}{\bar{B_1}}+\frac{\bar{B_4}}{\bar{B_1}})_{1234}$.\\
If
$\sum_{1234}^{}\bar{B}=\sum_{123}^{}\bar{B}=\sum_{124}^{}\bar{B}=\sum_{134}^{}\bar{B}$,
then \[(\bar{B_i})_{1234}=x_i (\bar{B_4})_{1234}+ y_i, i=1,2,3:\]
\[(x_1,y_1)=(\frac{(\frac{\bar{B_1}}{\bar{B_4}})_{134}(\frac{\bar{B_2}}{\bar{B_1}})_{123}+(\frac{\bar{B_1}}{\bar{B_4}})_{124}(\frac{\bar{B_3}}{\bar{B_1}})_{123}-1
         }{1+(\frac{\bar{B_2}}{\bar{B_1}})_{123}+(\frac{\bar{B_3}}{\bar{B_1}})_{123}},(\bar{B_1})_{123})\] \[(x_2,y_2)=(x_1(\frac{\bar{B_2}}{\bar{B_1}})_{123}-(\frac{\bar{B_1}}{\bar{B_4}})_{134}(\frac{\bar{B_2}}{\bar{B_1}})_{123}
         ,(\bar{B_2})_{123})\] \[(x_3,y_3)=(x_1(\frac{\bar{B_3}}{\bar{B_1}})_{123}-(\frac{\bar{B_1}}{\bar{B_4}})_{124}(\frac{\bar{B_3}}{\bar{B_1}})_{123},(\bar{B_3})_{123}).\]
\end{lemma}

\begin{theorem}\label{nounivconstant}
A universal constant does not exist for a unique fixed (variable)
weighted Fermat-Torricelli tree of degree four which is obtained
as a limiting tree structure from a family of variable weighted
Gauss-Menger trees (or full variable weighted Steiner trees) w.r.
to a fixed boundary convex quadrilateral $A_{1}A_{2}A_{3}A_{4},$
which depend on the variable weights $(\bar{B_{i}})_{1234}$ which
satisfy the dynamic plasticity equations

\begin{equation} \label{plastic1P4quadmod}
(\frac{\bar{B_2}}{\bar{B_1}})_{1234}=(\frac{\bar{B_2}}{\bar{B_1}})_{123}[1-(\frac{\bar{B_4}}{\bar{B_1}})_{1234}
(\frac{\bar{B_1}}{\bar{B_4}})_{134}],
\end{equation}
\begin{equation} \label{plastic2P5quadmod}
(\frac{\bar{B_3}}{\bar{B_1}})_{1234}=(\frac{\bar{B_3}}{\bar{B_1}})_{123}[1-(\frac{\bar{B_4}}{\bar{B_1}})_{1234}
(\frac{\bar{B_1}}{\bar{B_4}})_{124}],
\end{equation}
and
\begin{equation}\label{invcond4quadmod}
 (\bar{B_{1}})_{1234}+(\bar{B_{2}})_{1234}+(\bar{B_{3}})_{1234}+(\bar{B_{4}})_{1234}=c constant.
\end{equation}

or

\begin{eqnarray} \label{plastic14nmod}
(\bar{B_{1}})_{1234}^{2}+ (\bar{B_{2}})_{1234}^{2}+2
(\bar{B_{1}})_{1234}
(\bar{B_{2}})_{1234}\cos\alpha_{102}= \nonumber\\
(\bar{B_{3}})_{1234}^{2}+(\bar{B_{4}})_{1234}^{2}+2
(\bar{B_{3}})_{1234} (\bar{B_{4}})_{1234}\cos\alpha_{304},
\end{eqnarray}
\begin{eqnarray}\label{plastic24nmod}
(\bar{B_{1}})_{1234}^{2}+ (\bar{B_{4}})_{1234}^{2}+2
(\bar{B_{1}})_{1234} (\bar{B_{4}})_{1234}\cos\alpha_{104}=\nonumber\\
(\bar{B_{2}})_{1234}^{2}+ (\bar{B_{3}})_{1234}^{2}+2
(\bar{B_{2}})_{1234} (\bar{B_{3}})_{1234}\cos\alpha_{203},
\end{eqnarray}
and
\begin{equation}\label{invcond4nmod}
 (\bar{B_{1}})_{1234}+(\bar{B_{2}})_{1234}+(\bar{B_{3}})_{1234}+(\bar{B_{4}})_{1234}=c const.
\end{equation}

\end{theorem}

\begin{proof}

Suppose that we select four positive weights
$(\bar{B_{i}})_{1234}(0)$ which correspond the vertex $A_{i}$ of
the boundary convex quadrilateral $A_{1}A_{2}A_{3}A_{4},$ such
that the weighted floating inequalities (\ref{floatingcase}) of
Theorem~\ref{theor1} hold, in order to locate the 4wFT point
$A_{0}$ at the interior of $A_{1}A_{2}A_{3}A_{4}$ and particularly
at the interior of $\triangle A_{1}OA_{2},$ where $O$ is the
intersection of the diagonals $A_{1}A_{3}$ and $A_{2}A_{4},$
closer to the vertex $A_{1}.$

The location of the 3wFT points $A_{0}$ and $A_{0^{\prime}},$
respectively, is given by the relations
(Lemma~\ref{explicitsolution}):

\begin{equation}\label{base1nbar}
\cot{\varphi}=\frac{x_{G}
a_{12}+\bar{B_{4}}a_{14}\cos(\alpha_{214}-\alpha_{400^{\prime}})+\bar{B_{3}}a_{23}\cos(\alpha_{123}-\alpha_{30^{\prime}0})}{\bar{B_{4}}a_{14}\sin(\alpha_{214}-\alpha_{400^{\prime}})-\bar{B_{3}}a_{23}\sin(\alpha_{123}-\alpha_{30^{\prime}0})},
\end{equation}

\begin{equation}\label{a1solvesystemnbar}
a_{1}=\frac{a_{14}\sin(\alpha_{214}-\varphi-\alpha_{400^{\prime}})}{\sin(\alpha_{100^{\prime}}+\alpha_{400^{\prime}})}
\end{equation}

and

\begin{equation}\label{a2solvesystemnbar}
a_{2}=\frac{a_{23}\sin(\alpha_{123}+\varphi-\alpha_{300^{\prime}})}{\sin(\alpha_{20^{\prime}0}+\alpha_{30^{\prime}0})}
\end{equation}

where

\begin{equation}\label{ft01bar}
\alpha_{100^{\prime}}=\arccos\left(\frac{\bar{B_{4}}^{2}-\bar{B_{1}}^{2}-x_{G}^{2}}{2
\bar{B_{1}}x_{G}}\right),
\end{equation}
\begin{equation}\label{ft02bar}
\alpha_{0^{\prime}04}=\arccos\left(\frac{\bar{B_{1}}^{2}-\bar{B_{4}}^{2}-x_{G}^{2}}{2
\bar{B_{4}}x_{G}}\right),
\end{equation}
\begin{equation}\label{ft03bar}
\alpha_{104}=\arccos\left(\frac{x_{G}^{2}-\bar{B_{1}}^{2}-\bar{B_{4}}^{2}}{2
\bar{B_{1}}\bar{B_{4}}}\right),
\end{equation}

\begin{equation}\label{ft021bar}
\alpha_{00^{\prime}3}=\arccos\left(\frac{\bar{B_{2}}^{2}-\bar{B_{3}}^{2}-x_{G}^{2}}{2
\bar{B_{3}}x_{G}}\right),
\end{equation}
\begin{equation}\label{ft022bar}
\alpha_{00^{\prime}2}=\arccos\left(\frac{\bar{B_{3}}^{2}-x_{G}^{2}-\bar{B_{2}}^{2}}{2
x_{G}\bar{B_{2}}}\right),
\end{equation}

\begin{equation}\label{ft023bar}
\alpha_{20^{\prime}3}=\arccos\left(\frac{x_{G}^{2}-\bar{B_{2}}^{2}-\bar{B_{3}}^{2}}{2
\bar{B_{2}}\bar{B_{3}}}\right).
\end{equation}

and the variable weights $\bar{B_{i}}\equiv (\bar{B_{i}})_{1234}$
are taken from the system of equations (\ref{plastic1P4quadmod}),
(\ref{plastic2P5quadmod}), (\ref{invcond4quadmod}) or
(\ref{plastic14nmod}), (\ref{plastic24nmod}) and
(\ref{invcond4nmod}). Taking into account Lemma~\ref{lemlinearB4},
we express $(\bar{B_{i}})_{1234}$ as a function of
$\bar{B_{4}})_{1234}$

By taking into account the distance of $a_{12}$ from the line
defined by $A_{0}$ and $A_{0}^{\prime},$ and the distance of
$A_{2}$ from the line which passes through $A_{1}$ and is parallel
to the line defined by $A_{0}$ and $A_{0}^{\prime},$ we express
$l$ as a function w.r. to $\bar{B_{4}},$ $x_{G}$ and the five
Euclidean elements $a_{12},$ $a_{13},$ $a_{23},$ $a_{34},$
$a_{24},$ for $i=1,2,3,4.$

\begin{equation}\label{universalvaluesxG}
l=a_{1}\cos(\alpha_{100^{\prime}})+a_{2}\cos(\alpha_{20^{\prime}0})+a_{12}\cos(\varphi)
\end{equation}

By letting $l\equiv \epsilon$ a positive real number we derive a
nonlinear equation which depends only on the absorbing rate of the
generalized Gauss-Menger tree of degree three, where the 3wFT
point $A_{0}$ remains the same because the weights
$(\bar{B_{i}})_{1234}$ satisfy the dynamic plasticity equations of
the fixed variable weighted Fermat-Torricelli tree of degree four
concerning the same boundary quadrilateral $A_{1}A_{2}A_{3}A_{4}$
where the position of $A_{0}$ remains invariant in
$\mathbb{R}^{2}.$

By letting $l\equiv \epsilon_{i}$ a decreasing sequence which
converge to zero, $A_{0}^{\prime}$ will approach the fixed
position of $A_{0}.$

A universal constant $u_{c}$ may occur if $x_{G}(\epsilon_{k})\to
u_{c}$ for every positive real value of the variable weight
$(\bar{B_{4}})_{1234}.$ Thus, a weighted Fermat-Torricelli degree
four regarding the same boundary quadrilateral
$A_{1}A_{2}A_{3}A_{4}$ would exist with a constant absorbing rate
and could offer an analytical solution of the 4wFT problem in
$\mathbb{R}^{2}$ which is contradictory with the result of
Theorem~\ref{nonexist4FTsol}.
\end{proof}
The location of the 4wFT point for a convex quadrilateral or the
position of a weighted Fermat-Torricelli tree of degree four is
given by the following lemma which has been derived in
\cite{Zachos/Zou:88}
($(B_{1})_{1234}>(B_{2})_{1234}>(B_{3})_{1234}>(B_{4})_{1234}$):

\begin{lemma}{\cite[Formula~(5),(6),(7),(8),p.~413]{Zachos/Zou:88}}\label{location4ft}
The following system of equations allow us to compute the position
of the 4wFT point $A_{0}$ and provides a necessary condition to
locate it at the interior of the convex quadrilateral
$A_{1}A_{2}A_{3}A_{4}:$

\begin{equation} \label{eq:evquad3}
\cot(\alpha_{013})=\frac{\sin(\alpha_{213})-\cos(\alpha_{213})
\cot(\alpha_{102})- \frac{a_{31}}{a_{12}
}\cot(\alpha_{304}+\alpha_{401})}
{-\cos(\alpha_{213})-\sin(\alpha_{213}) \cot(a_{102})+
\frac{a_{31}}{a_{12} }},
\end{equation}
\begin{equation} \label{eq:evquad4}
\cot(\alpha_{013})=\frac{\sin(\alpha_{314})-\cos(\alpha_{314})
\cot(\alpha_{401})+ \frac{a_{31}}{a_{41}
}\cot(\alpha_{304}+\alpha_{401})}
{\cos(\alpha_{314})+\sin(\alpha_{314}) \cot(\alpha_{401})-
\frac{a_{31}}{a_{41} }},
\end{equation}

\begin{eqnarray} \label{eq:evquad2}
(B_3)_{1234}^2=(B_1)_{1234}^2+(B_2)_{1234}^2+(B_4)_{1234}^2+\nonumber\\
2(B_2)_{1234}(B_4)_{1234}\cos(\alpha_{401}+\alpha_{102})+2(B_1)_{1234}(B_2)_{1234}\cos(\alpha_{102})+\nonumber\\
+2(B_1)_{1234}(B_4)_{1234}\cos(\alpha_{401}),
\end{eqnarray}

\begin{equation} \label{eq:evquad1}
\cot(\alpha_{304}+\alpha_{401})=
\frac{(B_1)_{1234}+(B_2)_{1234}\cos(\alpha_{102})+(B_4)_{1234}\cos(\alpha_{401}))}{(B_4)_{1234}\sin(\alpha_{401})-(B_2)_{1234}\sin(\alpha_{102})},
\end{equation}

\begin{equation}\label{eq:evquad5}
\alpha_{203}=2\pi-\alpha_{102}-\alpha_{304}-\alpha_{401},
\end{equation}

\begin{equation}\label{a01quad4ft1}
a_{01}=a_{14}\frac{\sin(\alpha_{013}+\alpha_{314}+\alpha_{401})}{\sin\alpha_{401}}
\end{equation}

and

\begin{equation}\label{a04quad4ft2}
a_{04}=a_{14}\frac{\sin(\alpha_{013}+\alpha_{314})}{\sin\alpha_{401}}.
\end{equation}

\end{lemma}

\begin{example}\label{numuniversalminimumFTvalue2}
Given a rectangle $A_{1}A_{2}A_{3}A_{4}$ in $\mathbb{R}^{2}$ such
that: $A_{1}=(0,0),$ $A_{2}=(7,0),$ $A_{3}=(7,4),$ $A_{4}=(0,4),$
with side lengths $a_{12}=a_{34}=7$, $a_{23}=a_{41}=4$ and initial
positive weights $(B_{1})_{1234}=3,$ $(B_{2})_{1234}=2.5$
$(B_{3})_{1234}=1.7$ $(B_{4})=1.5$ which correspond to the
vertices $A_{1},$ $A_{2},$ $A_{3}$ and $A_{4},$ respectively. By
substituting these data in (\ref{eq:evquad3}), (\ref{eq:evquad4})
(\ref{eq:evquad2}), (\ref{eq:evquad1}), (\ref{eq:evquad5}) of
Lemma~\ref{location4ft} we obtain the location of the 4wFT point
$A_{0}=(2.8274502,1.2787811)$ which coincides with result obtained
by the Weiszfeld algorithm (\cite{Weis:37},\cite{Weis:09}) with 7
digit precision and we derive that:

\begin{eqnarray}\label{maindirections1}
\alpha_{102}=138.625^{\circ},
\alpha_{203}=50.1502^{\circ},\nonumber\\
\alpha_{304}=102.986^{\circ}, \alpha_{401}=68.2392^{\circ}.
\end{eqnarray}

\begin{figure}
\centering
\includegraphics[scale=0.75]{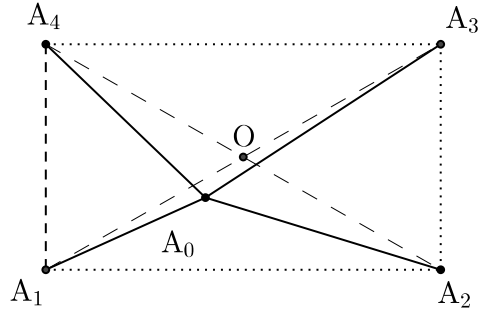}
\caption{A fixed weighted Fermat-Torricelli tree of degree four
having a universal Fermat-Torricelli minimum value $u_{FT}=\min
x_{G}=3.8088826$ with respect to a boundary weighted rectangle
taken from Example~2.}
\end{figure}

By substituting (\ref{maindirections1}) in
(\ref{plastic1P4quadmod}), (\ref{plastic2P5quadmod}) and
(\ref{invcond4quadmod}) or in (\ref{plastic14nmod}),
(\ref{plastic24nmod}) and (\ref{invcond4nmod}) we obtain the
dynamic plasticity equations of $A_{1}A_{2}A_{3}A_{4}:$

\begin{eqnarray}\label{dynamicplasticexam1}
(B_{1})_{1234}=4.2239621- 0.8159745 (B_4)_{1234},\nonumber\\
(B_{2})_{1234}=0.8393665+ 1.1070888 (B_4)_{1234},\nonumber\\
(B_{3})_{1234}=3.6366712- 1.2911143 (B_4)_{1234}.
\end{eqnarray}

Then, we  replace (\ref{dynamicplasticexam1})  for
$(B_{4})_{1234}=1.5,$ $(B_{4})_{1234}=1.2,$ $(B_{4})_{1234}=1.7,$
$(B_{4})_{1234}=1.7728955,$ in the objective function
(\ref{eq:B_1}) of a generalized Gauss tree of degree three, taking
into account (\ref{base1}), (\ref{a2solvesystem}) and
(\ref{a1solvesystem}) from Lemma~\ref{explicitsolution} and by
maximizing (\ref{eq:B_1}) w.r. to $x_{G},$ we derive the following
table (Fig.~2).

\begin{tabular}{|c|c|c|c|c|c|c|c|}
  \hline
  $(B_{1})_{1234}$ & $(B_{2})_{1234}$ & $(B_{3})_{1234}$ & $(B_{4})_{1234}$ & $x_{G}$ & f  \\
  \hline
  3& 2.5& 1.7& 1.5& 3.8192408 & 34.5746856\\
   3.2447927& 2.1678731& 2.0873328 & 1.2& 3.8543169&  34.6371118\\
  2.8368055& 2.7214176 & 1.4417756& 1.7 & 3.8096235&  34.5330567\\
  2.7773246 & 2.8021194& 1.3476592& 1.7728955 & 3.8088826 & 34.5178864\\

  \hline
\end{tabular}

By substituting $l=0.0000001$ and (\ref{dynamicplasticexam1}) in
(\ref{universalvaluesxG}) and by minimizing
(\ref{universalvaluesxG}) w.r. to the variables $x_{G}$ and
$(B_{4})_{1234},$ we obtain the universal minimum
Fermat-Torricelli value $u_{FT}=3.8088826$ for $(B_{4})_{1234}=
1.7728955$ (see also the above Table, Fig.~2).

The universal absorbing rate is
$\frac{u_{FT}}{\sum_{i=1}^{4}(B_{i})_{1234}}=\frac{3.8088826}{8.7}=0.4378025.$

\end{example}

\begin{example}\label{numuniversalminimumFTvalue3}
Given the same rectangle $A_{1}A_{2}A_{3}A_{4}$ in
$\mathbb{R}^{2}$ with the one considered in
Example~\ref{nounivconstant} and initial positive weights
$(B_{1})_{1234}=3.1,$ $(B_{2})_{1234}=2.3$ $(B_{3})_{1234}=1.7$
$(B_{4})=1.4$ which correspond to the vertices $A_{1},$ $A_{2},$
$A_{3}$ and $A_{4},$ respectively. By substituting these data in
(\ref{eq:evquad3}), (\ref{eq:evquad4}) (\ref{eq:evquad2}),
(\ref{eq:evquad1}), (\ref{eq:evquad5}) of Lemma~\ref{location4ft}
we obtain the location of the 4wFT point $A_{0}=(2.381487,
1.1855484)$ which coincides with result obtained by the Weiszfeld
algorithm (\cite{Weis:37},\cite{Weis:09}) with 7 digit precision
and we derive that:

\begin{eqnarray}\label{maindirections1new}
\alpha_{102}=139.138^{\circ},
\alpha_{203}=45.7542^{\circ},\nonumber\\
\alpha_{304}=98.8792^{\circ}, \alpha_{401}=76.2283^{\circ}.
\end{eqnarray}

\begin{figure}
\centering
\includegraphics[scale=0.75]{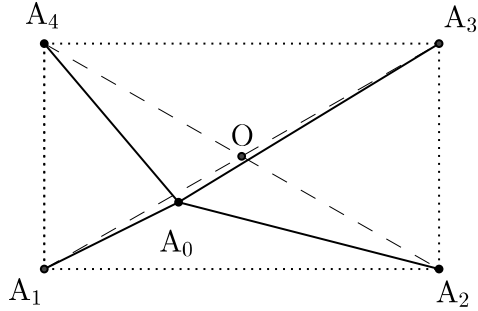}
\caption{A fixed weighted Fermat-Torricelli tree of degree four
having a universal Fermat-Torricelli minimum value $u_{FT}=\min
x_{G}=3.66326$ with respect to a boundary weighted rectangle taken
from Example~3.}
\end{figure}

By substituting (\ref{maindirections1new}) in
(\ref{plastic1P4quadmod}), (\ref{plastic2P5quadmod}) and
(\ref{invcond4quadmod}) or in (\ref{plastic14nmod}),
(\ref{plastic24nmod}) and (\ref{invcond4nmod}) we obtain the
dynamic plasticity equations of $A_{1}A_{2}A_{3}A_{4}:$

\begin{eqnarray}\label{dynamicplasticexam12}
(B_{1})_{1234}=4.1823652- 0.7731178 (B_4)_{1234},\nonumber\\
(B_{2})_{1234}=0.49794+ 1.2871855 (B_4)_{1234},\nonumber\\
(B_{3})_{1234}=3.8196947- 1.5140677(B_4)_{1234}.
\end{eqnarray}

By substituting $l=0.0000001$ and (\ref{dynamicplasticexam1}) in
(\ref{universalvaluesxG}) and by minimizing
(\ref{universalvaluesxG}) w.r. to the variables $x_{G}$ and
$(B_{4})_{1234},$ we obtain the universal minimum
Fermat-Torricelli value $u_{FT}=3.66326$ for $(B_{4})_{1234}=
1.8199325.$

The universal absorbing rate is
$\frac{u_{FT}}{\sum_{i=1}^{4}(B_{i})_{1234}}=\frac{3.66326}{8.5}=0.4309717$
(Fig.~3).

\end{example}


\begin{definition}\label{absorbingFermatTorricelli set}
A universal absorbing Fermat-Torricelli set with respect to a
fixed variable weighted Fermat-Torricelli tree of degree four is
the set of values of the Gauss variable $x_{G}$ which maximizes
the objective function of a generalized Gauss Menger tree of
degree three which is concentrated on the same weighted
Fermat-Torricelli tree of degree four, such that the variable
weights satisfy the dynamic plasticity equations of the weighted
boundary quadrilateral.
\end{definition}

\begin{definition}\label{Existenceuniversalminimumvalue}
We call a universal Fermat-Torricelli minimum value $u_{FT}$ the
minimum of the universal Fermat-Torricelli set regarding a fixed
variable weighted Fermat-Torricelli tree of degree four.
\end{definition}

\begin{remark}
The weighted Fermat-Torricelli tree of degree four taken from
example 2 has greater universal Fermat-Torricelli minimum value
$u_{FT}$ and universal absorbing rate than the weighted
Fermat-Torricelli tree of degree four w.r. to the same boundary
quadrilateral $A_{1}A_{2}A_{3}A_{4}$ with different weights
\end{remark}


\section{Steady trees and evolutionary trees for a boundary convex quadrilateral}\label{UniversalabsFTset}

Consider a universal Fermat-Torricelli set $U$ and minimum value
$u_{FT}$ which corresponds to a fixed variable weighted
Fermat-Torricelli tree of degree four regarding a boundary convex
quadrilateral, which is derived as a limiting tree structure from
a generalized Gauss tree of degree.

\begin{definition}\label{Steadytree}A steady tree of degree four is a
weighted Fermat-Torricelli tree of degree four having as storage
at the 4wFT point $A_{0}$ a positive quantity less than $u_{FT}.$
\end{definition}

\begin{definition}\label{Evolutiarytree}An evolutionary tree of degree three is
a generalized Gauss Menger tree of degree three which is derived
by a weighted Fermat-Torricelli tree of degree four having as
storage at the 4wFT point $A_{0}$ a positive quantity equal or
greater than $u_{FT}$ and then decreases by an absorbing rate
$0<a_{G}<u_{FT}.$
\end{definition}

\begin{example}
Consider the same weighted Fermar-Torricelli tree four that we
have obtained in Example~2 for the boundary rectangle
$A_{1}A_{2}A_{3}A_{4}.$ A steady tree is a weighted
Fermat-Torricelli tree of degree four with storage level at
$A_{0}$ less than $u_{FT}=3.8088826.$

When the storage quantity at $A_{0}$ reaches  $u_{FT}=3.8088826,$
it stimulates the steady tree which starts to evolute
(Fig.~\ref{fig3}). Thus, the universal minimum Fermat-Torricelli
value unlocks the evolution of a generalized Gauss-Menger tree
which could be derived by a weighted Fermat-Torricelli tree of
degree four for the same boundary rectangle (or convex
quadrilateral).

For instance, if the storage level reaches
$u=3.8543169>u_{FT}=3.8088826$ and spends $a_{G}=0.5,$ then we
derive an evolutionary generalized Gauss tree having
$x_{G}=3.3543169$ (Fig.~\ref{fig4}).

\begin{figure}
\centering
\includegraphics[scale=0.75]{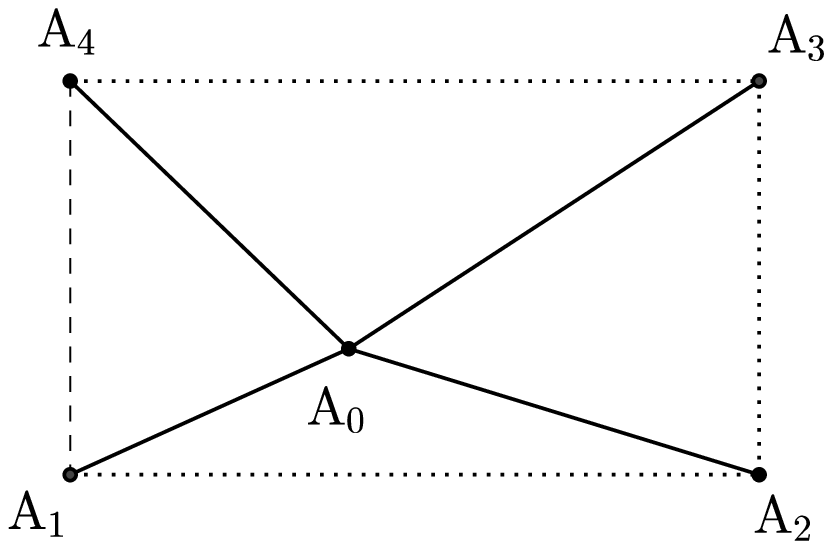}
\caption{An evolutionary tree of degree four having an initial
storage level $3.8543169$ and absorbing rate  $a_{G}=0$ with
respect to a boundary weighted rectangle taken from Example~2 for
$(B_{1})_{1234}=3.2447927,$ $(B_{2})_{1234}=2.1678731,$
$(B_{3})_{1234}=2.0873328,$ $(B_{4})_{1234}=1.2.$}\label{fig3}
\end{figure}

\begin{figure}
\centering
\includegraphics[scale=0.75]{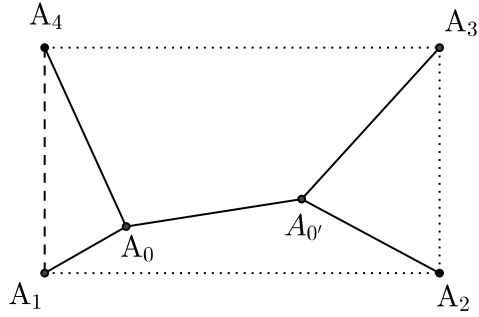}
\caption{An evolutionary tree of degree three having an initial
storage level $3.8543169$ and absorbing rate  $a_{G}=0.5$ with
respect to a boundary weighted rectangle taken from Example~2 for
$(B_{1})_{1234}=3.2447927,$ $(B_{2})_{1234}=2.1678731,$
$(B_{3})_{1234}=2.0873328,$ $(B_{4})_{1234}=1.2.$}\label{fig4}
\end{figure}

The evolutionary Gauss-Menger tree is obtained by substituting
$x_{G}=3.3543169,$ $(B_{1})_{1234}=3.2447927,$
$(B_{2})_{1234}=2.1678731,$ $(B_{3})_{1234}=2.0873328,$
$(B_{4})_{1234}=1.2$ in  (\ref{base1}), (\ref{a2solvesystem}) and
(\ref{a1solvesystem}) from Lemma~\ref{explicitsolution}. Thus, we
get:
\begin{eqnarray}
a_{1}=1.6642065, a_{2}= 2.7738702, a_{3}= 3.6321319,\nonumber\\
a_{4}=3.4873166, l=3.3543169.
\end{eqnarray}


Suppose that the storage level reaches $u=3.82>u_{FT}=3.8088826$
and spends $a_{G}=0.2,$ then we derive another evolutionary
weighted Fermat-Torricelli tree of degree four and a generalized
Gauss tree of degree three having $x_{G}=3.62$
(Figs.~\ref{fig5},~\ref{fig6},\ref{fig7}).

By substituting $u=x_{G}=3.82$ and the dynamic plasticity
equations (\ref{dynamicplasticexam1}) taken from Example~2 in
(\ref{eq:B_1}), the maximum of (\ref{eq:B_1}) w.r. to
$(B_{4})_{1234}$ yields $(B_{4})_{1234}=1.4901507$ or
$(B_{4})_{1234}=2.0556426.$ By letting $(B_{4})_{1234}=1.4901507$
in (\ref{dynamicplasticexam1}), we get:
$(B_{1})_{1234}=3.0080371,$ $(B_{2})_{1234}=2.4890958,$
$(B_{3})_{1234}=1.7127149,$

and by letting $(B_{4})_{1234}=2.0556426$ in
(\ref{dynamicplasticexam1}), we get: $(B_{1})_{1234}=2.5466101,$
$(B_{2})_{1234}=3.1151456,$ $(B_{3})_{1234}=0.9826002.$

\begin{figure}
\centering
\includegraphics[scale=0.75]{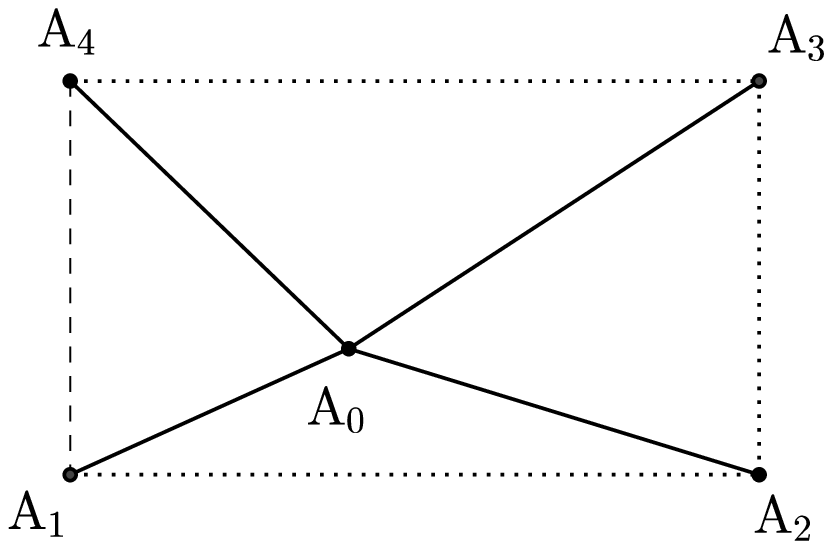}
\caption{An evolutionary tree of degree four having an initial
storage level $3.82$ and absorbing rate  $a_{G}=0$ with respect to
a boundary weighted rectangle taken from Example~2 for
$(B_{1})_{1234}=3.0080371,$ $(B_{2})_{1234}=2.4890958,$
$(B_{3})_{1234}=1.7127149,$
$(B_{4})_{1234}=1.4901507.$}\label{fig5}
\end{figure}

\begin{figure}
\centering
\includegraphics[scale=0.75]{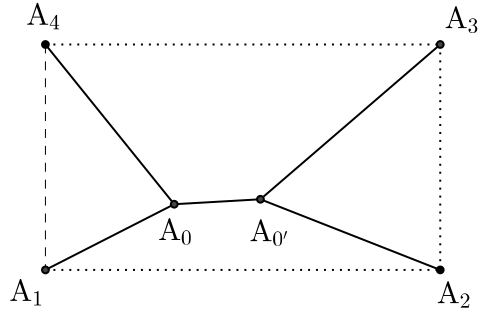}
\caption{An evolutionary tree of degree three having an initial
storage level $3.82$ and absorbing rate  $a_{G}=0.2$ with respect
to a boundary weighted rectangle taken from Example~2
$(B_{1})_{1234}=3.0080371,$ $(B_{2})_{1234}=2.4890958,$
$(B_{3})_{1234}=1.7127149,$
$(B_{4})_{1234}=1.4901507.$}\label{fig6}
\end{figure}

The first evolutionary Gauss-Menger tree is obtained by
substituting $x_{G}=3.62,$ $(B_{1})_{1234}=3.0080371,$
$(B_{2})_{1234}=2.4890958,$ $(B_{3})_{1234}=1.7127149,$
$(B_{4})_{1234}=1.4901507$ in  (\ref{base1}),
(\ref{a2solvesystem}) and (\ref{a1solvesystem}) from
Lemma~\ref{explicitsolution} (Fig.~\ref{fig6}). Thus, we get:
\begin{eqnarray}
a_{1}=2.5638686, a_{2}=3.4255328 , a_{3}=4.2080591 ,\nonumber\\
a_{4}=3.6397828, l=1.5309344.
\end{eqnarray}

The second evolutionary Gauss-Menger tree is obtained by
substituting $x_{G}=3.62,$ $(B_{1})_{1234}=2.5466101,$
$(B_{2})_{1234}=3.1151456,$ $(B_{3})_{1234}=0.9826002.$
$(B_{4})_{1234}=2.0556426.$ in  (\ref{base1}),
(\ref{a2solvesystem}) and (\ref{a1solvesystem}) from
Lemma~\ref{explicitsolution} (Fig.~\ref{fig7}). Thus, we get:
\begin{eqnarray}
a_{1}=2.6836315, a_{2}=3.0204233 , a_{3}=4.0857226 ,\nonumber\\
a_{4}=3.6424502, l=1.8001622.
\end{eqnarray}

\begin{figure}
\centering
\includegraphics[scale=0.75]{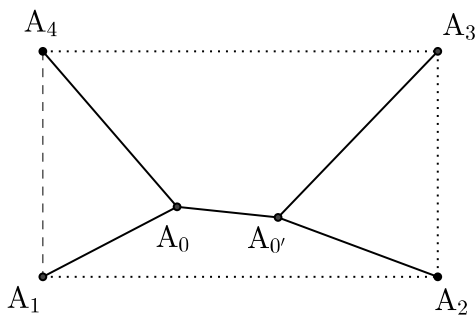}
\caption{An evolutionary tree of degree three having an initial
storage level $3.82$ and absorbing rate  $a_{G}=0.2$ with respect
to a boundary weighted rectangle taken from Example~2
$(B_{1})_{1234}=2.5466101,$ $(B_{2})_{1234}=3.1151456,$
$(B_{3})_{1234}=0.9826002.$
$(B_{4})_{1234}=2.0556426.$}\label{fig7}
\end{figure}

\end{example}


\begin{thebibliography}{99}

\bibitem{BolMa/So:99} V. Boltyanski, H. Martini, V. Soltan, \emph{Geometric Methods and Optimization Problems}, Kluwer, Dordrecht-Boston-London, 1999.
\bibitem{Ci} D. Cieslik, Shortest connectivity. An introduction with applications in
phylogeny, Springer-Verlag, New York (2005).
\bibitem{Cour/Rob:51} R. Courant and H. Robbins, \emph{What is Mathematics?} Oxford University Press., New York, 1951.

\bibitem{ENGELBRECHt:1877} E.Engelbrecht, \emph{Planimetrischer Lehrsatz}. Arch. Math. Phys.
\textbf{60}(1877), 447--448.

\bibitem{GilbertPollak:68} E.N. Gilbert and H.O. Pollak, \emph{Steiner Minimal
trees}, SIAM Journal on Applied Mathematics.\textbf{16} (1968),
1-29.

\bibitem{Gue/Tes:02} S. Gueron and R. Tessler, \emph{The Fermat-Steiner problem}, Amer. Math. Monthly \textbf{109} (2002), 443-451.

\bibitem{JarnikKossler:34} V. Jarnik and R. Kossler, \emph{O minimalnich grafeth obeahujicich n
danijch bodu}, Cas. Pest. Mat. a Fys., 63 (1934), 223--235.



\bibitem{IvanovTuzhilin:95}A. O. Ivanov and A. A. Tuzhilin,
\emph{Weighted minimal binary trees}. (Russian) Uspekhi Mat. Nauk
50 (1995), no. 3(303), 155--156; translation in Russian Math.
Surveys 50 (1995), no. 3, 623--624.
\bibitem{IvanovTuzhilin:01} A. O. Ivanov and A. A. Tuzhilin,
\emph{Branching solutions to one-dimensional variational
problems}. World Scientific Publishing Co., Inc., River Edge, NJ,
2001.
\bibitem{Kup/Mar:97} Y.S. Kupitz and H. Martini, \emph{Geometric aspects of the generalized Fermat-Torricelli problem},
Bolyai Society Mathematical Studies, \textbf{6}, (1997) , 55-127.

\bibitem{Weis:37} E. Weiszfeld, \emph{Sur le point lequel la somme des distances de n points donnes est minimum}, Tohoku Math. J. \textbf{43},
(1937), 355--386.

\bibitem{Weis:09} E. Weiszfeld, \emph{On the point for which the sum of the distances to $n$ given
points is minimum}, Ann. Oper. Res. \textbf{167} (2009), 7-41.




\bibitem{Zachos/Zou:88} A.N. Zachos and G. Zouzoulas, \emph{An evolutionary structure of convex quadrilaterals}, J. Convex Anal., \textbf{15}, no. 2 (2008), 411--426.
\bibitem{Zachos/Zou:13} A.N. Zachos and G. Zouzoulas, \emph{An evolutionary structure of convex quadrilaterals. Part II.}, J. Convex Anal., \textbf{20}, no. 2 (2013),  483--493.
\bibitem{Zachos:98} A.N. Zachos, \emph{A weighted Steiner minimal tree for convex quadrilaterals on the two-dimensional K-plane}, J. Convex Anal., \textbf{18}, no. 1 (2011), 139--152.
\bibitem{Zachos:14} A.N. Zachos, \emph{A Plasticity Principle of Convex
Quadrilaterals on a Convex Surface of Bounded Specific Curvature
},Acta. Appl. Math. (2014),\textbf{129}, no. 1, 81--134.
\bibitem{Zachos:15} A.N. Zachos, \emph{Solving a generalized Gauss
problem}, Mediterr. J. Math. 12 (2015), 1069-1083.



\end{thebibliography}
\end{document}